\documentclass[12pt]{amsart}[2000]

\title{A dichotomy for central sequence algebras}

\author{Ilijas Farah}
\address{Department of Mathematics and Statistics\\
	York University\\
	4700 Keele Street\\
	North York, Ontario\\ Canada, M3J 1P3\\
	and 
	Ma\-te\-ma\-ti\-\v cki Institut SANU\\
	Kneza Mihaila 36\\
	11\,000 Beograd, p.p. 367\\
	Serbia}
\email{ifarah@yorku.ca}
\urladdr{https://ifarah.mathstats.yorku.ca}
\thanks{I.F. is partially supported by NSERC}

\author{Ilan Hirshberg}
\address{Department of Mathematics, Ben Gurion University of the Negev, P.O.B. 653, Be'er \indent Sheva 84105, Israel}
\email{ilan@math.bgu.ac.il}

\usepackage{amssymb}
\usepackage{amsmath}
\usepackage{amsthm}
\usepackage{xypic}

\theoremstyle{plain}

\newtheorem{Thm}{Theorem}%[section]

\newtheorem{Lemma}[Thm]{Lemma}

\theoremstyle{definition}
\newtheorem{Def}[Thm]{Definition}

\newcommand{\N}{{\mathbb N}}

\newcommand{\C}{{\mathbb C}}

\newcommand{\prodM}{\prod_{n=1}^\infty M_n}
\newcommand{\prodCM}{\prod_{n=1}^\infty CM_n}
\newcommand{\cstar}{$\mathrm{C}^*$}
\newcommand{\cst}{\mathrm{C}^*}
\newcommand{\cU}{\mathcal U}
\newcommand{\bbN}{\mathbb N}
\newcommand{\bt}{\mathbf t}
\DeclareMathOperator{\Diag}{Diag}

\begin{document}
\begin{abstract}
We prove that the central sequence algebra of a separable \cstar-algebra is either subhomogeneous or non-exact, confirming a conjecture of Enders and Shulman. We also prove analogous dichotomy for other massive \cstar-algebras. 
\end{abstract}
\maketitle

By refining a result of \cite{ando2016non}, in \cite{enders2022commutativity} it was proven that for a separable \cstar-algebra $A$ the central sequence algebra $A_\cU\cap A'$ (where $\cU$ is a nonprincipal  ultrafilter on $\bbN$, fixed throughout,  and $A_\cU$ is the \cstar-algebra ultrapower) ) is subhomogeneous if and only if $A$ satisfies Fell's condition of some finite order (\cite[Section 3]{enders2022commutativity}).  Tatiana Shulman conjectured (personal communication) that for a separable $A$,  $A_\cU\cap A'$ is either subhomogeneous or non-type I. 
We provide a strong confirmation of this conjecture, as well as an analogous statement for corona algebras. Namely, we show:

\begin{Thm}\label{T.1}
	If $A$ is a separable \cstar-algebra and $B$ is a separable \cstar-subalgebra of $A_{\cU}$,  then $A_\cU\cap B'$ is either subhomogeneous or non-exact. 
\end{Thm}

\begin{Thm}\label{T.corona}
	Let $A$ be a  $\sigma$-unital, non-unital, \cstar-algebra and let $B$ be a separable \cstar-subalgebra of the corona algebra $Q(A) = M(A)/A$. Then $Q(A) \cap B'$ is either subhomogeneous or non-exact. 
\end{Thm}

The proofs are related to Kirchberg's proof that an ultrapower of a \cstar-algebra $A$ is nuclear if and only if it is subhomogeneous  (\cite[Proposition 3.14.1]{Muenster}). 
Readers familiar with the notion of countable degree-1 saturation (see \cite[\S 15]{Fa:STCstar}) may jump to the conclusion that the dichotomy exhibited in  Theorems~\ref{T.1} and~\ref{T.corona} holds for all countably degree-1 saturated \cstar-algebras. While we conjecture that this is the case, we haven't been able to prove it.

%\begin{Thm} \label{T.2} Suppose that $C$ is (i) an ultrapower of a separable \cstar-algebra, (ii) the corona of a $\sigma$-unital, non-unital, \cstar-algebra, or (iii) the relative commutant of a separable \cstar-subalgebra of a \cstar-algebra as in (i) or (ii). Then $C$ is is either subhomogeneous or not exact. 
%\end{Thm}

For information on types and saturation see \cite[\S\S  15--16]{Fa:STCstar}. We only need the following simple variant (\cite[Definition~15.2.1]{Fa:STCstar}). 

\begin{Def} A \emph{quantifier-free type} in variables $x_n$, for $n\in \bbN$, is a set of conditions of the form 
	\[
	\|p(x_0,\dots, x_{n-1})\|=r
	\] 
	where $p$ is a *-polynomial in non-commuting variables. \footnote{Readers familiar with \cite[\S 16]{Fa:STCstar} will notice that the general definition allows polynomials with coefficients in a \cstar-algebra $C$; this is (quantifier-free) type \emph{over $C$}. This generality will not be needed here.}A quantifier-free type $\bt$ is \emph{realized} in a \cstar-algebra $C$ if there are $(a_n)_{n\in \bbN}$ in the unit ball of $C$ such that $\|p(a_0,\dots, a_n)\|=r$ for every condition $\|p(x_0,\dots, x_n)\|=r$ in $\bt$.  It is  \emph{approximately satisfiable} in a \cstar-algebra $C$ if for every finite subset  $\bt_0$ of $\bt$ and for every $\varepsilon>0$ there are $(a_n)_{n\in \bbN}$ in the unit ball of $C$ such that $\left | \|p(a_0,\dots, a_n)\|-r \right |<\varepsilon$ for every condition  $\|p(x_0,\dots, x_n)\|=r$ in $\bt_0$.  
	A \cstar-algebra $B$ is \emph{countably quantifier-free saturated} if every consistent quantifier free type $\bt(\bar{x})$ in countably many variables is realized in $B$.
\end{Def}

Thus we define a \cstar-algebra $C$ to be countably quantifier-free saturated if every countable quantifier-free type \emph{in countably many variables} consistent with its theory is realized in $C$. This is slightly different from the definition normally used in the literature (e.g., \cite[\S 16]{Fa:STCstar}) where only types with finitely many variables are considered. 
In case of full saturation (where arbitrary formulas are allowed), a simple trick shows that the two variants of the definition are equivalent. In case of quantifier-free types this is not obvious, since the trick involves formulas with $\inf$-quantifiers, however this can still be done for central sequence algebras. The proof is essentially identical to that of  \cite[Proposition 16.5.3]{Fa:STCstar}, using countably many variables in place of finitely many variables, so we omit the details. We record it as a the following lemma:
\begin{Lemma}\label{lem:countable_qf_saturation}
	Let $A$ be a separable \cstar-algebra, and let $C$ be a separable \cstar-subalgebra of $A_{\cU}$. Then $A_{\cU} \cap C'$ is countably quantifier-free saturated (for types with countably many variables). 
\end{Lemma}

	Suppose that $A$ is a separable \cstar-algebra, fix a dense subset $(a_n)_{n \in \N}$ of the unit ball of $A$, and  let $\bt_A(\bar{x})$ be the quantifier-free type of  $(a_n)_{n \in \N}$. 	More precisely, fix an enumeration of of *-polynomials in non-commuting variables in $(x_n)_{n \in \N}$, $p_n(x_0,\dots, x_{n-1})$, for $n\in \bbN$, with coefficients in $\mathbb{Q}[i]$ (or any other countable dense subset of $\C$).  (For convenience of notation, we can assume that the variables of the $n$-th polynomial are included in $(x_j)_{j<n}$.) Consider the quantifier-free type $\bt_A(x_n: n\in \N)$ whose conditions are, with $r_n=\|p_n(a_0,\dots, a_{n-1})\|$:
\[
\|p_n(x_0,\dots, x_{n-1})\|=r_n.
\]
The type $\bt_A$ depends on the choice of $a_n$, but this makes no difference. (Instead of $\bt_A$ one could consider the entire atomic diagram $\Diag(A)$, see \cite[p. 15]{Muenster}, but we have no use for it here.)

\begin{Lemma}
Let $A$ and $B$ be \cstar-algebras, and suppose $A$ is separable. Then
\begin{enumerate}
\item $A$ embeds into $B$ if and only if $\bt_A$ is realized in $B$.
\item    $A$ embeds into $B_\cU$ if and only if the type $\bt_A$  is approximately satisfiable in $B$.
\item If $C$ is a separable \cstar-subalgebra of $B_\cU$ then $A$ embeds in $B_\cU\cap C'$ if and only if $\bt_A$ is approximately satisfiable in $B_\cU\cap C'$. 
 \end{enumerate}
\end{Lemma}

\begin{proof}
	The first two assertions are trivial. The third assertion follows from countable saturation of ultraproducts (\cite[Theorem~16.4.1]{Fa:STCstar}) and  \cite[Proposition~16.5.3]{Fa:STCstar}. 
\end{proof}

Note that $B_\cU\cap C'$ is not necessarily fully countably saturated, even if both $B$ and $C$ are separable (\cite[Proposition~6.3]{farah2023obstructions}). 

Write  $M_n$ for the algebra of $n\times n$ complex matrices and denote the cone by $CM_n = C_0( (0,1] , M_n)$. 

\begin{Lemma}\label{L.nonexact}
	There exist separable non-exact subalgebras of $\prod_{n=1}^{\infty}M_{n}$ and of $\prod_{n=1}^{\infty} CM_{n}$.
\end{Lemma}

\begin{proof} The full group \cstar-algebra $\cst(F_2)$ is not exact by \cite[Corollary~3.7.12]{BrOz:C*}. Since it is separable and residually finite-dimensional (RFD), it is  isomorphic to a subalgebra of $\prodM$ by \cite[Theorem 7.4.1]{BrOz:C*}. 
Since $M_n$ is  a quotient of $CM_n$ for all $n$, $\prodM$ is a quotient of $\prodCM$. 
By \cite[Corollary~9.4.3]{BrOz:C*}, quotients of exact \cstar-algebras are exact (on p. 304 of \cite{BrOz:C*} the authors state that `At present, this is one of the hardest \cstar-results, ever').  Therefore $\prodCM$ is not exact. 
By \cite[Theorem~10.2.6]{BrOz:C*}, an inductive limit of exact \cstar-algebras is exact, and therefore $\prodCM$ has a non-exact separable \cstar-subalgebra. 
\end{proof}

\begin{Lemma}\label{L.CM}
	Let $n \in \N$, let $A$ be a \cstar-algebra and let $\pi \colon CM_n \to A$ be a nonzero homomorphism. Then either the image $\pi$ is isomorphic to $CM_n$ or its image contains a subalgebra isomorphic to $M_n$.
\end{Lemma}
\begin{proof}
Let $h \in CM_n$ be given by $h(t) = t 1_{M_n}$. Then $\pi ( CM_n ) \cong C_0 (\sigma ( \pi(h) ) \smallsetminus \{0\} , M_n)$ (where $\sigma (\pi(h))$ is the spectrum of $\pi(h)$). If $\sigma (\pi(h)) = [0,r]$ for some $r \in (0,1]$, then $\pi (CM_n) \cong CM_n$. If not, then pick $r_0 \in (0, \| h \| ) \smallsetminus \sigma (\pi(h))$. Then $\rho \colon M_n \to \pi (CM_n)$ by $\rho(a) =  \xi_{ (r_0,1) \cap \sigma (\pi(h)) } \otimes a$ is an injective homomorphism.
\end{proof}

\begin{Thm}\label{thm:main}
	Let $B$ be a countably quantifier-free saturated \cstar-algebra. Either $B$ is $n$-subhomogeneous for some $n \in \N$ or $B$ is not exact.
\end{Thm}
\begin{proof}
	Assume that $B$ is not $n$-subhomogeneous for any~	$n$. We show that $B$ is not exact.
	Because $\bigoplus_{k=1}^n M_k$ can be embedded as a (block-diagonal) subalgebra of $M_{n(n+1)/2}$, if there exists a subalgebra of $B$ isomorphic to $M_{n(n+1)/2}$ then we can embed $\bigoplus_{k=1}^n M_k$ in $B$, and  if there exists a subalgebra of $B$ isomorphic to $CM_{n(n+1)/2}$ then we can embed $\bigoplus_{k=1}^n CM_k$ in $B$. 
	For any $n$, because $B$ is not $[ n(n+1)/2 -1 ]$-subhomogeneous, there exists an irreducible representation $\pi$ of $B$ on a Hilbert space of dimension greater or equal to $n(n+1)/2$. Let $\rho_0 \colon M_{n (n+1)/2} \to \pi(B)''$ be an embedding. Using the Kaplansky density theorem for c.p.c. order zero maps (see for example \cite[Lemma 1.1]{HKW}), there exists a net of c.p.c. order zero maps $\rho_{\nu} \colon M_{n (n+1)/2} \to \pi(B)$ which converges weak$^*$ to $\rho_0$. In particular, there exists a nonzero order zero map $\rho \colon M_{n (n+1)/2} \to \pi(B)$. Using the one-to-one correspondence between order zero maps and homomorphisms from cones (\cite[Corollary~4.1]{WiZac:Completely}), by slight abuse of notation, we think of $\rho$ as a homomorphism $\rho \colon CM_{n (n+1)/2} \to \pi(B)$. Using projectivity of cones of matrix algebras (\cite[Theorem 10.2.1]{Loring}), we can lift $\rho$ to a homomorphism $\tilde{\rho} \colon CM_{n (n+1)/2} \to B$.
	
	Thus, for every $n$, there is a nonzero homomorphism $\psi_n \colon CM_{n (n+1)/2} \to B$. If the images of infinitely many of those homomorphisms are isomorphic to $CM_{n (n+1)/2}$, then it means that for any $n$, we have an embedding of  $\bigoplus_{k=1}^n CM_k$ in $B$. If not, then by Lemma~\ref{L.CM} we have an embedding of $\bigoplus_{k=1}^n M_k$ in $B$.

Using Lemma~\ref{L.nonexact}, 	let $A_1$ be some fixed separable non-exact subalgebra of $\prod_{n=1}^{\infty}CM_n$, and let $A_2$ be some fixed separable non-exact subalgebra of $\prod_{n=1}^{\infty}M_n$. 
	
	Assume that we are in the first case, that is, that for any $n$, we have an embedding of  $\bigoplus_{k=1}^n CM_k$ in $B$. 
%%%%%	
 For any finite subset $\bt_0 \subset \bt$ and for any $\varepsilon>0$, because $\bt_0$ is realized in $\prod_{n=1}^{\infty}CM_n$, there exists $N_{\varepsilon} \in \N$ such that $\bt_0$ is realized up to $\varepsilon$ in $\bigoplus_{n=1}^{N_{\varepsilon}}CM_n$. 
 This is because with $\pi_m\colon \prod_{n=1}^{\infty} CM_n\to \prod_{n<m} CM_n$ denoting the projection map, for every $a\in \prod_{n=1}^{\infty} CM_n$ we have  $\|a\|=\lim_{m \to \infty} \|\pi_m(a)\|=\sup_m \| \pi_m(a) \|$. 

	Because $B$ is countably quantified-free saturated, $\bt(\bar{x})$ is realized in $B$, but this means that there exists an embedding of $A_1$ in $B$. Because $A_1$ is not exact, and exactness passes to subalgebras, $B$ is not exact either, as claimed.

	The proof in the second case is identical, replacing $A_1$ with $A_2$.
\end{proof}

\begin{proof}[Proof of Theorem~\ref{T.1}]  This follows from a combination of Theorem~\ref{thm:main} and Lemma~\ref{lem:countable_qf_saturation}. 
\end{proof}

We now prove Theorem~\ref{T.corona}, which is the analogue of Theorem~\ref{T.1} for relative commutants in corona algebras. Those are not countably quantified-free saturated, however one obtains a similar statement, with a very similar proof.
%\begin{Thm}\label{thm.2}
%	Let $A$ be a $\sigma$-unital \cstar-algebra, and let $C$ be a separable \cstar-subalgebra of the corona algebra $Q(A)$. Let $B = Q(A) \cap C'$.  Either $B$ is $n$-subhomogeneous for some $n \in \N$ or $B$ is not exact.
%\end{Thm}
\begin{proof}[Proof of Theorem~\ref{T.corona}]
Assume that $B$ is not $n$-subhomogeneous for any $n$. Arguing as in the proof of Theorem~\ref{thm:main}, for every $n$ we can fix an order zero map $\psi_n \colon M_n \to B$ of norm $1$. Let $\pi \colon M(A) \to Q(A) = M(A)/A$ be the quotient map, and set $\tilde{B} = \pi^{-1}  (  B ) $.  We can lift every $\psi_n$ to an order zero map $\tilde{\psi}_n \colon M_n \to \tilde{B}$. Let $f_1 \leq f_2 \leq f_3 \leq \cdots$ be an  approximate identity for $A$ satisfying $f_{n+1}f_n = f_n$ for all $n$ and such that for every $a$ in the range of  some $\tilde\psi_j$ we have $\lim_{n\to \infty} \|[f_n,a]\|=0$ (\cite[Proposition~1.9.3]{Fa:STCstar}). By passing to a subsequence if needed, we can assume without loss of generality that for every $x \in M_n$ we have
\[
\| ( f_{2n+1} - f_{2n} ) \tilde{\psi}_n(x)  ( f_{2n+1} - f_{2n} ) \| > \left ( 1-\frac{1}{n} \right )\|x\|
\, .
\] 
Set $e_n =  f_{2n+1} - f_{2n}$. Note that $e_1,e_2,e_3,\ldots$ are pairwise orthogonal. By slight abuse of notation and using \cite[Corollary~4.1]{WiZac:Completely} as in the proof of Theorem~\ref{thm:main},  we think of $\tilde{\psi}_n$ and $\psi_n$ as homomorphisms from $CM_n$. By Lemma~\ref{L.CM}, passing to a subsequence, and taking subalgebras, we can assume that either all of them are injective, or that the images of $\psi_n$ all have a copy of $M_n$ in their image, in which case we can replace them by homomorphisms from $M_n$ to $B$. Define $\Psi \colon \prod_{n=1}^{\infty} CM_n \to \tilde{B}$ to be $\Psi (x_1,x_2,x_3,\cdots) = \sum_{n=1}^{\infty} e_n \tilde{\psi}_n(x_n) e_n$, with the sum taken to converge in the strict topology. Then $\Psi$ is a completely positive map, and since $f_n$ approximately commute with the range of $\tilde\psi_n$ for all $n$, the composition $\pi \circ \Psi$ has order zero. Its range is therefore isomorphic either to $ \prod_{n=1}^{\infty} CM_n / \bigoplus_{n=1}^{\infty} CM_n$ or to $ \prod_{n=1}^{\infty} M_n / \bigoplus_{n=1}^{\infty} M_n$. As those are not exact, neither is $B$.
\end{proof}

In the case of the Calkin algebra, a stronger result was obtained in \cite[Lemma~8]{enders2022commutativity}, who proved that the relative commutant of any separable \cstar-subalgebra contains a copy of $B(H)$. 

We don't know whether the analogue of Theorem~\ref{T.1} holds for Kirchberg's modified central sequence algebra, $F(B,A)$ (see \cite[Definition 1.1]{kirchberg_abel}). 
 We note that  $F(B,A)$ (or even $F(A) = F(A,A)$) in general is not countably quantifier-free saturated, so one cannot use the same proof as the one used for Theorem~\ref{T.1}. For example, consider $A = c_0(\N)$. Then \cite[Proposition~1.9 (5)]{kirchberg_abel} implies that  
\[
F(A) \cong A' \cap M ( \overline{   A A_{\cU}  A} )\cong  l^{\infty}(\N). 
\]
Since this is an infinite-dimensional von Neumann algebra,  it is not countably quantifier-free saturated (\cite[Exercise~15.6.4]{Fa:STCstar}).

We point out that the central sequences of nonseparable \cstar-algebras are not as well-behaved; for up-to-date information on the case of ${\mathcal B}(H)$ see \cite{chetcuti2023commutant}. 

\subsection*{Acknowledgments} This work was done during the 2023 Fields Institute Thematic Program on Operator Algebras and Applications. We thank Tatiana Shulman for communicating the problem solved in Theorem~\ref{T.1}. 

\bibliographystyle{plain}
\bibliography{nonexact_bib}

\begin{thebibliography}{10}

\bibitem{ando2016non}
H.~Ando and E.~Kirchberg.
\newblock Non-commutativity of the central sequence algebra for separable
  non-type {I} {$C^*$}-algebras.
\newblock {\em J. London Math. Soc.}, 94(1):280--294, 2016.

\bibitem{BrOz:C*}
N.~Brown and N.~Ozawa.
\newblock {\em {\cstar}-algebras and finite-dimensional approximations},
  volume~88 of {\em Graduate Studies in Mathematics}.
\newblock American Mathematical Society, Providence, RI, 2008.

\bibitem{chetcuti2023commutant}
E.~Chetcuti and B.~Zamora-Aviles.
\newblock On the commutant of {$B (H)$} in its ultrapower.
\newblock {\em Israel J. Math.}, 255(1):423--451, 2023.

\bibitem{enders2022commutativity}
D.~Enders and T.~Shulman.
\newblock Commutativity of central sequence algebras.
\newblock {\em Adv. Math.}, 401:108262, 2022.

\bibitem{Fa:STCstar}
I.~Farah.
\newblock {\em Combinatorial Set Theory and \cstar-algebras}.
\newblock Springer Monographs in Mathematics. Springer, 2019.

\bibitem{Muenster}
I.~Farah, B.~Hart, M.~Lupini, L.~Robert, A.~Tikuisis, A.~Vignati, and
  W.~Winter.
\newblock Model theory of \cstar-algebras.
\newblock {\em Mem. Amer. Math. Soc.}, 271(1324):viii+127, 2021.

\bibitem{farah2023obstructions}
I.~Farah and A.~Vignati.
\newblock Obstructions to countable saturation in corona algebras.
\newblock {\em Proc. Amer. Math. Soc.}, 151(03):1285--1300, 2023.

\bibitem{HKW}
I.~Hirshberg, E.~Kirchberg, and S.~White.
\newblock Decomposable approximations of nuclear {$C^*$}-algebras.
\newblock {\em Adv. Math.}, 230(3):1029--1039, 2012.

\bibitem{kirchberg_abel}
E.~Kirchberg.
\newblock Central sequences in {$C^*$}-algebras and strongly purely infinite
  algebras.
\newblock In {\em Operator {A}lgebras: {T}he {A}bel {S}ymposium 2004}, volume~1
  of {\em Abel Symp.}, pages 175--231. Springer, Berlin, 2006.

\bibitem{Loring}
T.~A. Loring.
\newblock {\em Lifting solutions to perturbing problems in {$C^*$}-algebras},
  volume~8 of {\em Fields Institute Monographs}.
\newblock American Mathematical Society, Providence, RI, 1997.

\bibitem{WiZac:Completely}
W.~Winter and J.~Zacharias.
\newblock Completely positive maps of order zero.
\newblock {\em M\"unster J. Math.}, 2:311--324, 2009.

\end{thebibliography}

\end{document}